\documentclass[reqno,11pt,draft]{amsart}
\usepackage{amsmath,amssymb,amsfonts}

\usepackage{tikz}
\usetikzlibrary{arrows}
\usepackage{verbatim}

\title[A compact generator on formal schemes]{On the existence of a compact generator on the derived category of a noetherian formal scheme}

\author[L. Alonso]{Leovigildo Alonso Tarr\'{\i}o}
\address[L. A. T.]{Departamento de \'Alxebra\\
Facultade de Matem\'a\-ticas\\
Universidade de Santiago de Compostela\\
E-15782  Santiago de Compostela, Spain}
\email{leo.alonso@usc.es}

\author[A. Jerem\'{\i}as]{Ana Jerem\'{\i}as L\'opez}
\address[A. J. L.]{Departamento de \'Alxebra\\
Facultade de Matem\'a\-ticas\\
Universidade de Santiago de Compostela\\
E-15782  Santiago de Compostela, Spain}
\email{ana.jeremias@usc.es}

\author[M. P\'erez]{Marta P\'erez Rodr\'{\i}guez}
\address[M. P. R.]{Departamento de Matem\'a\-ticas\\
Esc. Sup. de Enx. Inform\'atica,
Campus de Ourense\\
Universidade de Vigo\\
E-32004 Ourense, Spain}
\email{martapr@uvigo.es}

\author[M. J. Vale]{Mar\'{\i}a J. Vale Gonsalves}
\address[M. J. V.]{Departamento de \'Alxebra\\
Facultade de Matem\'a\-ticas\\
Universidade de Santiago de Compostela\\
E-15782  Santiago de Compostela, Spain}
\email{mj.vale@usc.es}

\thanks{This work has been partially supported by 
Spain's MEC and E.U.'s FEDER research projects MTM2005-05754 
and MTM2008-03465.
}


\date{\today}

\theoremstyle{plain}
\newtheorem{thm}{Theorem}[section]
\newtheorem{lem}[thm]{Lemma}
\newtheorem{cor}[thm]{Corollary}
\newtheorem{prop}[thm]{Proposition}

\theoremstyle{remark}
\newtheorem*{rem}{Remark}

\theoremstyle{definition}

\newtheorem{cosa}[thm]{}

\numberwithin{equation}{thm}

\newcommand{\CC}{{\mathcal C}}

\newcommand{\CE}{{\mathcal E}}
\newcommand{\CF}{{\mathcal F}}
\newcommand{\CG}{{\mathcal G}}
\newcommand{\CH}{{\mathcal H}}
\newcommand{\CI}{{\mathcal I}}

\newcommand{\CK}{{\mathcal K}}
\newcommand{\CL}{\mathcal{L}}

\newcommand{\CO}{\mathcal{O}}
\newcommand{\CP}{\mathcal{P}}
\newcommand{\CQ}{\mathcal{Q}}

\newcommand{\FS}{\mathfrak S}

\newcommand{\FU}{\mathfrak U}
\newcommand{\FV}{\mathfrak V}

\newcommand{\FX}{\mathfrak X}
\newcommand{\FY}{\mathfrak Y}
\newcommand{\FZ}{\mathfrak Z}

\newcommand{\SC}{\mathsf{C}}

\newcommand{\D}{\boldsymbol{\mathsf{D}}}

\newcommand{\R}{\boldsymbol{\mathsf{R}}}
\newcommand{\T}{\boldsymbol{\mathsf{T}}}
\newcommand{\A}{\mathsf{A}}

\newcommand{\cc}{\mathsf{c}}

\newcommand{\ts}{\mathsf{t}}
\newcommand{\md}{\text{-}\mathsf{Mod}}

\newcommand{\qc}{\mathsf{qc}}
\newcommand{\qct}{\mathsf{qct}}

\newcommand{\ZZ}{\mathbb{Z}}

\newcommand{\ia}{{\mathfrak a}}

\newcommand{\ip}{{\mathfrak p}}

\newcommand{\kk}{\kappa}

\newcommand{\dirlim}[1]{\begin{array}[t]{c} {\rm lim}\\[-7.5 pt]
 {\longrightarrow} \\[-7.5 pt] {\scriptstyle {#1}} \end{array}}

\newcommand{\lto}{\longrightarrow}
\newcommand{\xto}{\xrightarrow}

\newcommand{\inc}{\hookrightarrow}
\DeclareMathOperator{\iso}{\tilde{\to}}

\DeclareMathOperator{\liso}{\tilde{\lto}}

\newcommand{\dimp}{\Leftrightarrow}

\newcommand{\tr}{\triangle}

\DeclareMathOperator{\Hom}{Hom}
\DeclareMathOperator{\shom}{\CH\mathit{om}}

\DeclareMathOperator{\spec}{Spec}
\DeclareMathOperator{\spf}{Spf}
\DeclareMathOperator{\supp}{Supp}

\DeclareMathOperator{\id}{id}

\newcommand{\ie}{{\it i.e.\/} }
\newcommand{\cfr}{{\it cf.\/} }
\newcommand{\lc}{{\it loc.cit.\/} }

\begin{document}

\begin{abstract} In this paper, we prove that for a noetherian formal scheme $\FX$, its derived category of sheaves of modules with quasi-coherent torsion homologies $\D_\qct(\FX)$ is generated by a single compact object. In an appendix we prove that the category of compact objects in $\D_\qct(\FX)$ is skeletally small.
\end{abstract}

\maketitle

\section*{Introduction}

Formal schemes have played a basic role in algebraic geometry since its introduction by Grothendieck in \cite{GB}. They were intended to express in the language of schemes the theory of holomorphic functions of Zariski, and later they played an outstanding role in Grothendieck-Lefschetz theory. Its cohomology, however was poorly understood, in part due to the lack of an appropriate derived category of coefficients.

The situation has changed since the work \cite{dfs} where the basic cohomology theory including Grothendieck duality was developed. In \lc the category of quasi-coherent torsion coefficients was introduced, where torsion is related to the canonical adic topology in the structure sheaf on the formal scheme. The category makes it possible the development of a theory of cohomology for non-adic maps of formal schemes. Further results were obtained in \cite{LNS} and \cite{fgm}.

In these developments the basic ingredient is the derived category of sheaves of modules with quasi-coherent torsion homologies $\D_\qct(\FX)$ where $(\FX, \CO_\FX)$ is a noetherian formal scheme. Recent results have been the determination of its Bousfield localizations in connection with the underlying geometric structure \cite{AJSB}. Also it has been shown in \cite{ahst} that whenever $\FX$ is semi-separated $\D_\qct(\FX)$ is a stable homotopy category in the sense of Hovey, Palmieri and Strickland \cite{hps}.

One of the main features of this structure is the existence of good generators. It has been known for some time that derived categories of smooth projective varieties have a single compact generator. Compact objects correspond to the classical notion of perfect complex studied by Grothendieck's school \cite{I1}. Surprisingly, Bondal and Van den Bergh proved in \cite{bb} that the derived category of sheaves of modules with quasi-coherent homologies $\D_\qc(X)$ ---where $(X, \CO_X)$ is a quasi-compact and quasi-separated scheme--- is generated by a single perfect complex. Later, Lipman and Neeman obtained a different proof \cite{LN} with additional precisions on the generator under the additional hypothesis that $X$ is separated.

In our previous paper \cite[Proposition 6.14]{ahst} we showed that $\D_\qct(\FX)$ is compactly generated for a noetherian formal scheme $\FX$. The proof there gives no control on how many compact objects are needed to generate the category. In this paper we show that $\D_\qct(\FX)$ possesses a single compact generator (Theorem \ref{one}). The feasibility of getting this result was remarked to the authors by Van den Bergh. In the case in which $\FX$ is of pseudo-finite type over a field, say, this opens the possibility of studying $\FX$ through the differential graded algebra of endomorphisms of its generator, but this is a question for later study.

Our results depend on the machinery developed in \cite{dfs} and \cite{ahst}. In the first section we collect a few general notions and establish some notations for convenience of the readers. In the second section we give the proof of the theorem. Proposition \ref{propb} is  an intermediate result on the relationship between the derived categories of sheaves of modules with quasi-coherent torsion homologies on a formal scheme, a closed formal subscheme and its open complement. It is used in a crucial way in the proof of our main result.

In the last section, we give a result that is used in the proof of Theorem \ref{one} but that has also an independent interest, namely, for a noetherian formal scheme $(\FX, \CO_\FX)$ the category of perfect complexes in $\D_\qct(\FX)$ is skeletally small, \ie the isomorphy classes of perfect complexes possess a representative in a certain set. This allows for constructions of invariants based on this set of classes for formal schemes, like Picard groups, $K$-theory, etc.
\section{Preliminaires}

We assume that the reader is familiar with the basic theory of formal schemes  as is explained in \cite[\S 10]{ega1}.
We will recall briefly some definitions and results about locally noetherian formal schemes. 

\begin{cosa}  \label{equiv}
The functors
\begin{equation*} 
A \leadsto \FX := \spf(A) \qquad \mathrm{and} \qquad
\FX \leadsto A := \Gamma(\FX, \CO_{\FX})
\end{equation*}
define a duality between the category of adic noetherian rings and the category of affine noetherian formal schemes that generalizes the duality between the categories of rings and affine schemes (see \cite[(10.2.2), (10.4.6)]{ega1}). Let $I$ be an ideal of definition of $A$. The formal scheme $\spf(A)$ is the completion of the scheme $\spec(A)$ along the closed subscheme $V(I)$. They are related through the completion map that we usually denote
\[
\kk \colon \spf(A) \lto \spec(A).
\]
\end{cosa}


\begin{cosa}  \label{defadic}
A morphism $f\colon\FX \to \FY$ of locally noetherian formal schemes is \emph{adic}  if there exists an Ideal of  definition $\CK$ of  $\FY$ such that $f^{*}(\CK)\CO_{\FX}$ is an Ideal of  definition of  $\FX$. Note that if there exists an Ideal of definition $\CK$ of $\FY$ such that $f^*(\CK)\CO_{\FX}$ is an Ideal of definition of $\FX$, then the same holds true for all Ideals of definition of $\FY$. See \cite[\S 10.12.]{ega1}.
\end{cosa}

\begin{cosa} \label{defnenc} 
Let $\FX$ be a locally noetherian formal scheme. Given $\CI \subset \CO_{\FX}$, a coherent Ideal, $\FX' := \supp(\CO_{\FX}/\CI)$ is a closed subset and  $(\FX', (\CO_{\FX}/\CI)|_{\FX'})$ is a noetherian formal scheme. We will say that $\FX'$ is the   \emph{closed (formal) subscheme} of  $\FX$ defined by $\CI$. See \cite[\S 10.14.]{ega1}. On the other hand, if  $\FU \subset \FX$ is an open subset, it holds that $(\FU,\CO_{\FX}|_{\FU})$ is a noetherian formal scheme  and we say that $\FU$ is an \emph{open subscheme  of  $\FX$} (see \cite[(10.4.4)]{ega1}).

A morphism $f:\FZ \to \FX$  is a \emph{closed immersion} (\emph{open immersion}) if there exists $\FZ' \subset \FX$ closed (open, respectively) formal subscheme such that  $f$ factors as
\[
\FZ \xto{g} \FZ' \inc \FX
\]
where $g$ is an isomorphism. Closed and open inmersions are adic morphisms.
\end{cosa}

\begin{cosa}
Let $(\FX, \CO_\FX)$ be a noetherian formal scheme with an ideal of definition  $\CI$. Let $\A(\FX)$ be the category of all $\CO_\FX$-modules and $\varGamma'_\FX \colon \A(\FX) \to \A(\FX)$ the functor defined by
\[ 
\varGamma'_\FX \CF := \dirlim{n > 0} \shom_{\CO_\FX}(\CO_\FX/\CI^n, \CF)
\] 
for $\CF \in \A(\FX)$. It does not depend on $\CI$, only on the topology defined by $\CI$ in the sheaf of rings $\CO_{\FX}$.  Let $\A_\ts(\FX)$ be the full subcategory of $\A(\FX)$ whose objects are those $\CF$ such that
$\varGamma'_\FX \CF = \CF$; it is a subcategory of $\A(\FX)$ closed for kernels, cokernels and extensions (\ie \emph{plump}, as defined in \cite[beginning of \S 1]{dfs}). Let $\A_\qc(\FX)$ be the subcategory of  $\A(\FX)$ formed by quasi-coherent sheaves (it makes sense for any ringed space, \cfr \cite[\textbf{0}, \S 5]{ega1}). We will denote by $\A_\qct(\FX) : = \A_\ts(\FX) \cap \A_\qc(\FX)$. It is again a plump
subcategory of $\A(\FX)$, see \cite[Corollary 5.1.3]{dfs} where it was introduced. The category $\A_\qct(\FX)$ is a Grothendieck category, \ie it has a generator and exact filtered direct limits \cite[Proposition 1.8]{toh} (see \cite[Lemma 6.1]{ahst}).

From the fact that $\A_\qct(\FX)$ is plump in $\A(\FX)$ it follows that the full subcategory of the derived category $\D(\FX) := \D(\A(\FX))$ formed by complexes whose homology lies in $\A_\qct(\FX)$ is a triangulated subcategory of  $\D(\FX)$. We will denote it by $\D_\qct(\FX)$. This category will be the main object of study of the present work.
\end{cosa}

\begin{cosa}
Let $\FX$ be a noetherian formal scheme. A complex $\CE^\bullet \in \D(\FX)$ is called \emph{perfect} if for every $x \in \FX$ there is an open neighborhood $\FU$ of $x$ and a bounded complex of locally-free finite type $\CO_\FU$-Modules $\CF^\bullet$ together with an isomorphism $\CF^\bullet \iso \CE^\bullet|_\FU$ in $\D(\FU)$ (\cfr \cite[D\'efinition 4.7.]{I1}).

Let $\T$ be a triangulated category with coproducts. An object $E$ of $\T$ is called \emph{compact} if the functor $\Hom_{\T}(E,-)$ commutes with arbitrary coproducts. By \cite[Proposition 6.12.]{ahst} the compact objects in $\D_\qct(\FX)$ are the perfect complexes.
\end{cosa}

\begin{cosa} \label{genadd}
Let $\T$ be as before and $S = \{E_{\alpha} \, / \, \alpha \in L\}$ a set of objects of $\T$. We say that $S$ is a \emph{set of generators} of $\T$ if, for an object $M \in \T$, $\Hom_{\T}(E_{\alpha}[i], M) = 0$, for all $\alpha \in L$ and for all $i \in \ZZ$, implies $M = 0$. The existence of a single generator is equivalent to the existence of a set of generators taking as single generator the coproduct of all the objects in the family. If the set $S$ is formed by compact objects then $S$ is a set of generators of $\T$ if and only if the smallest triangulated subcategory closed for coproducts that contains $S$ is all of $\T$, see \cite[Lemma 3.2]{Ngd}.
\end{cosa}

\section{The main theorem}

We start with some results that will be used throughout the proof of Theorem \ref{one}.

\begin{prop} \label{afin}
Let $\FX = \spf(A)$ be an affine formal scheme such that $A$ is an $I$-adic noetherian ring. The category $\D_\qct(\FX)$ is generated by a perfect complex.
\end{prop}

\begin{proof}
Take the complex of sheaves associated to a Koszul complex built on a system of generators of the ideal $I$, see \cite[Proposition 6.10]{ahst}. 
\end{proof}

\begin{rem}
This is the affine case of our main Theorem \ref{one}. It is also the first step in its proof.
\end{rem}

Let $\D_{I}(A)$ be the full subcategory of the derived category of $A$-modules formed by complexes whose homologies are $I$-torsion. In the proof of the previous Proposition, the equivalence of categories between the category $\D_\qct(\FX)$ and $\D_{I}(A)$, that follows from \cite[Proposition 5.2.4]{dfs}, is used in a decisive way.

Let $V(I) = \{\ip \in \spec(A) \,/\, I \subset \ip\}$ \ie the set of zeros of the ideal $I$ in $X := \spec(A)$ and $\D_{\qc,V(I)}(X)$ be the derived category of quasi-coherent sheaves in $X$ supported on $V(I)$. Let $\kk \colon \FX \to X$ be the canonical completion map. Consider the following composition of equivalences of derived categories
\[
\D_\qct(\FX) \xto{\,\,\,\kk_*\,\,\,} 
\D_{\qc,V(I)}(X) \xto{\R\Gamma(X, -)} 
\D_{I}(A).
\]
Let us denote this equivalence by $\phi \colon \D_\qct(\FX) \lto \D_{I}(A)$. Note that a quasi-inverse functor is the composition
\[
\D_{I}(A) \xto{\,\,\,(-)^{\sim}\,\,\,} 
\D_{\qc,V(I)}(X) \xto{\,\,\,\kk^*\,\,\,} 
\D_\qct(\FX).
\]
This composition will be denoted by $\psi \colon \D_{I}(A) \lto \D_\qct(\FX)$. 

Let $\FZ$ be a closed formal subscheme of a noetherian formal scheme $\FX$. We will denote by $\D_{\qct, \FZ}(\FX)$ the full subcategory of $\D_\qct(\FX)$ formed by the complexes $\CF^\bullet$ such that $\CH^p(\CF^\bullet)|_{\FX \setminus \FZ} = 0$ for all $p \in \ZZ$.

\begin{lem}\label{lemaa}
Let $\ia$ be an ideal of the $I$-adic ring $A$. Let $\FX = \spf(A)$ and $\FZ = \spf(A/\ia)$ the corresponding closed subscheme. The previous equivalence restricts to the subcategories $\D_{\qct, \FZ}(\FX)$ and $\D_{I+\ia}(A)$.
\end{lem}

\begin{proof}
Let $Z := V(I+\ia) \subset X$ where $X = \spec(A)$. It is clear from the previous discussion that the categories $\D_{I+\ia}(A)$ and $\D_{\qc,Z}(X)$ are equivalent through the quasi-inverse functors $(-)^{\sim}$ and $\R\Gamma(X, -)$. 

Let us deal with the general case $\FX = \spf(A)$ and check that $\kk_*$ takes $\D_{\qct, \FZ}(\FX)$ into $\D_{\qc,V(I+\ia)}(X)$. Indeed, we will see that if $\CF$ is a $(I+\ia)^\tr$-torsion quasi-coherent $A^\tr$-module\footnote{We recall that for a finite-type $A$-module $M$ we denote by $M^\tr$ its associated coherent sheaf on $\spf(A)$, \cite[(10.10.1)]{ega1}.} then $\kk_*(\CF)$ is a quasi-coherent $\widetilde{A}$-module supported on $Z$. If $\CF|_{\FX \setminus \FZ} = 0$ then $\CF_x = 0$ for all $x \in {\FX \setminus \FZ}$. But then it follows that $\kk_*\CF_y = 0$ for all $y \in X \setminus Z$. 

It remains to check that if $\CG$ is a quasi-coherent $\CO_X$-module supported on $Z$ then $\supp(\kk^*\CG) \subset \FZ$. Notice that for any $y \in \FX$,
\[
(\kk^*\CG)_y = \CG_{\kk(y)} \otimes_{\CO_{X,\kk(y)}} \CO_{\FX,y}.
\] 
Therefore if $\CG_x = 0$ for all $x \in X \setminus V(I+\ia)$, then $(\kk^*\CG)_y = 0$ for all $y \in \FX \setminus\FZ = \kk^{-1}(X \setminus Z)$.
\end{proof}

\begin{cosa}
Let $\T$ be a triangulated category and let $\CL$ be a full triangulated subcategory of $\T$. It is possible to construct after Verdier a new triangulated category denoted $\T/\CL$ by inverting the morphisms $A \to B$ in $\T$ such that if $A \to B \to C \to A[1]$ is a distinguished triangle then $C \in \CL$. This process can be performed through a calculus of fractions, see \cite[\S 2.1]{Ntc}. There is a canonical functor $q \colon \T \to \T/\CL$ that carries every object to itself. Note that the objects in $\CL$ go to zero by $q$ and that $q$ is universal for triangulated functors from $\T$ with this property.
\end{cosa}

\begin{prop}\label{propb}
Let $\FX$ be a noetherian formal scheme and $\FU$ an open formal subscheme. Let $\FZ := \FX \setminus \FU$. We have an equivalence of categories:
\[
\frac{\D_\qct(\FX)}{\D_{\qct, \FZ}(\FX)} \cong \D_\qct(\FU).
\]
\end{prop}

\begin{proof}
Let $\alpha \colon \FU \to \FX$ be the canonical immersion. Since $\alpha^*(\D_{\qct, \FZ}(\FX)) = 0$, there is a functor $\overline{\alpha}$ such that the diagram of functors
\[
 \begin{tikzpicture}
       \draw[white] (0cm,2cm) -- +(0: \linewidth)
      node (G) [black, pos = 0.3] {$\D_\qct(\FX)$}
      node (H) [black, pos = 0.7, scale=1.2] {$\frac{\D_\qct(\FX)}{\D_{\qct, \FZ}(\FX)}$};
      \draw[white] (0cm,0.5cm) -- +(0: \linewidth)
      node (E) [black, pos = 0.5] {$\D_\qct(\FU)$};
      \draw [->] (G) -- (H) node[above, midway, sloped, scale=0.75]{$q$};
      \draw [->] (G) -- (E) node[auto, swap, midway, scale=0.75]{$\alpha^*$};
      \draw [->] (H) -- (E) node[auto, midway, scale=0.75]{$\overline{\alpha}$};
  \end{tikzpicture}
\]
commutes. On the other hand it holds that $\R\alpha_*(\D_\qct(\FU)) \subset \D_\qct(\FX)$ by \cite[Proposition 5.2.6]{dfs}. Consider the composition
\[
\D_\qct(\FU) \xto{\R\alpha_*} \D_\qct(\FX) \overset{q}{\lto}
\frac{\D_\qct(\FX)}{\D_{\qct, \FZ}(\FX)}.
\]
We will see that $\overline{\alpha}$ and $q \circ \R\alpha_*$ are quasi-inverse functors. Note first that
\[
\overline{\alpha} \circ q \circ \R\alpha_* =
\alpha^* \circ \R\alpha_* =
\id_{\D_\qct(\FU)}.
\]
To finish the proof, we have to check that $q \circ \R\alpha_* \circ \overline{\alpha}$ is isomorphic to the identity functor on $\frac{\D_\qct(\FX)}{\D_{\qct, \FZ}(\FX)}$. Let $\CF^\bullet \in \D_\qct(\FX)$ and consider the localization triangle \cite[2.1.(1)]{AJSB}
\[
\R\varGamma_{\FZ}\CF^\bullet \lto 
\CF^\bullet \lto 
\R\alpha_*\alpha^*\CF^\bullet \overset{+}{\lto}.
\]
We have that $\R\varGamma_{\FZ}\CF^\bullet \in \D_{\qct, \FZ}(\FX)$ therefore $q(\R\varGamma_{\FZ}\CF^\bullet) = 0$ and we obtain an isomorphism $q\CF^\bullet \iso q\R\alpha_*\alpha^*\CF^\bullet$. But it follows that $q\CF^\bullet \iso q\R\alpha_*\overline{\alpha} (q\CF^\bullet)$, and as every object of ${\D_\qct(\FX)}/{\D_{\qct, \FZ}(\FX)}$ is of the form $q\CF^\bullet$ for some $\CF^\bullet \in \D_\qct(\FX)$, we have proved our claim and the proposition.
\end{proof}

\begin{cosa}
Let $\T$ be a triangulated category with coproducts and let $\CL$ be a full triangulated subcategory of $\T$. We will denote by 
\[
\CL^{\perp} = \{ Z \in \T \,/\, \Hom_{\T}(Y,Z) = 0 \text{ for every } Y \in \CL \}.
\] 
The full subcategory $\CL^{\perp}$ is a triangulated subcategory of $\T$. If $\{E_{\alpha} \, / \, \alpha \in L\}$
is a set of objects of $\T$, we will denote by $\langle E_{\alpha} \, / \, \alpha \in L \rangle$ the smallest full triangulated subcategory of $\T$ stable for coproducts containing $\{E_{\alpha} \, / \, \alpha \in L\}$. Observe that $\{E_{\alpha} \, / \, \alpha \in L\}^{\perp} = \langle E_{\alpha} \, / \, \alpha \in L \rangle^{\perp}$. A full triangulated subcategory of $\T$ stable for coproducts is called a localizing subcategory of $\T$. 
If $\CL$ is a localizing subcategory of $\T$, the objects of $\CL^{\perp}$ are called $\CL$-local objects.
\end{cosa}

\begin{cor}\label{bous}
In the setting of the previous proposition, the functor $q$ induces an equivalence of categories
\[
\widetilde{q} \colon \D_{\qct, \FZ}(\FX)^\perp \lto 
\frac{\D_\qct(\FX)}{\D_{\qct, \FZ}(\FX)}
\]
whose quasi-inverse is induced by $\R\alpha_* \circ \overline{\alpha}$.
\end{cor}

\begin{proof}
Combine \cite[Proposition 1.6, (ii) $\dimp$ (iii)]{AJS} with the previouos Proposition.
\end{proof}

\begin{thm} \label{one}
Let $(\FX, \CO_\FX)$ be a noetherian formal scheme. The category $\D_\qct(\FX)$ is generated by a perfect complex.
\end{thm}

\begin{proof}
Let us argue by induction in the number of affine open formal subschemes needed to cover $\FX$. Denote this number by $n(\FX)$. Let $n := n(\FX)$, we write $\FX = \FU_1 \cup \dots \cup \FU_n$ with each $\FU_i$ an affine formal scheme for every $i \in \{1, \dots, n\}$. The case $n = 1$ is covered by Proposition \ref{afin}.

Denote $\FU = \FU_1$ and $\FY = \FU_2 \cup \dots \cup \FU_n$. We can (and will) suppose by induction hypothesis that $\D_\qct(\FY)$ is generated by a perfect complex, say $\CE^\bullet$. Let $\FV :=  \FU \cap \FY$ and consider the following cartesian diagram:
\[
 \begin{tikzpicture}
       \draw[white] (0cm,2.5cm) -- +(0: \linewidth)
      node (G) [black, pos = 0.4] {$\FV$}
      node (H) [black, pos = 0.6] {$\FU$};
      \draw[white] (0cm,0.5cm) -- +(0: \linewidth)
      node (E) [black, pos = 0.4] {$\FY$}
      node (F) [black, pos = 0.6] {$\FX$};
      \draw [->] (G) -- (H) node[above, midway, sloped, scale=0.75]{$\alpha$};
      \draw [->] (G) -- (E) node[left, midway, scale=0.75]{$\beta$};
      \draw [->] (H) -- (F) node[right, midway, scale=0.75]{$\gamma$};
      \draw [->] (E) -- (F) node[below, midway, scale=0.75]{$\delta$};
  \end{tikzpicture}
\]
where all maps are the canonical open embeddings. 

Put $\FZ := \FX \setminus \FY = \FU \setminus \FV$ with $\FU = \spf(A)$ for a certain noetherian $I$-adic ring $A$. The formal scheme $\FZ$ is a closed formal subscheme of $\FU$. Therefore $\FZ$ is affine and there exists an ideal $\ia$ of $A$ such that $\FZ = \spf(A/\ia)$. Note that $A/\ia$ is an $I(A/\ia)$-adic ring and its underlying topological space is identified with the set of zeros of the ideal $I+\ia$ in $\spec(A/I)$. Let $Q^\bullet$ be a Koszul complex built on a set of generators of $I+\ia$. According to \cite[Proposition 6.1]{BN} (see, alternatively, \cite[Proposition 6.1]{dg}) it generates $\D_{I+\ia}(A)$. By the equivalence of Lemma \ref{lemaa}, the perfect complex $\CQ^\bullet := \kappa^*\widetilde{Q^\bullet}$ generates $\D_{\qct, \FZ}(\FU)$ where $\kk \colon \spf(A) \lto \spec(A)$ denotes the completion map, as usual. We have that:
\begin{align}
\delta^*\R\gamma_*\CQ^\bullet &= \R\beta_*\alpha^*\CQ^\bullet = 0 
    &\text{ (by open base change) }\label{bcQ}\\
\gamma^*\R\gamma_*\CQ^\bullet &= \CQ^\bullet & \label{bcQ2}
\end{align}
But $\{\FY, \FU\}$ constitute a covering of $\FX$, therefore the previous equalities imply that the complex $\R\gamma_*\CQ^\bullet$ is perfect.
Summing up, the categories $\D_\qct(\FU)$ and $\D_{\qct, \FZ}(\FU)$ are compactly generated. 

Denote by $\T^\cc$ the subcategory of compact objects of a triangulated category $\T$. By Proposition \ref{propb} we have an equivalence $\frac{\D_\qct(\FU)}{\D_{\qct, \FZ}(\FU)} \cong \D_\qct(\FV)$. Applying Thomason-Neeman localization theorem \cite[Theorem 2.1.]{Ntty}  or \cite[Theorem 2.1.]{Ngd} we see that $\frac{\D_\qct(\FU)^\cc}{\D_{\qct, \FZ}(\FU)^\cc}$ is a full subcategory of  $\left(\frac{\D_\qct(\FU)}{\D_{\qct, \FZ}(\FU)}\right)^\cc \cong \D_\qct(\FV)^\cc$ and moreover the latter is the smallest thick subcategory of $\D_\qct(\FV)$ containing the former. Now $\beta^*\CE^\bullet$  is a perfect complex and by \cite[Proposition 6.12]{ahst} it is a compact object. By the aforementioned equivalence of triangulated categories it follows that 
\[
q\R\alpha_*\beta^*\CE^\bullet \in \left(\frac{\D_\qct(\FU)}{\D_{\qct, \FZ}(\FU)}\right)^\cc,
\]
where $q \colon \D_\qct(\FU) \to {\D_\qct(\FU)}/{\D_{\qct, \FZ}(\FU)}$ is the equivalence of Proposition \ref{propb} for $\alpha \colon \FV \inc \FU$. 
Since $\D_\qct(\FV)^\cc$ is skelletally small by Theorem \ref{sksm} below and having in mind \cite[Corollary 3.2.3]{bb}, there is an object $\CF^\bullet \in \D_\qct(\FU)^\cc$ such that 
\[q\CF^\bullet \cong q\R\alpha_*\beta^*(\CE^\bullet\oplus\CE^\bullet[1]) 
\in
\frac{\D_\qct(\FU)^\cc}{\D_{\qct, \FZ}(\FU)^\cc}.\] 
Thus
$\overline{\alpha}q\CF^\bullet \cong \overline{\alpha}q\R\alpha_*\beta^*(\CE^\bullet\oplus\CE^\bullet[1])$, but $\alpha^* = \overline{\alpha}q$ and it follows that
we have an isomorphism $\alpha^*\CF^\bullet \iso \beta^*(\CE^\bullet\oplus\CE^\bullet[1])$ in $\D_\qct(\FV)$. Again by \cite[Proposition 6.12]{ahst} the complex $\CF^\bullet$ is perfect because it is compact.
We are going to \emph{glue} the complexes $\CF^\bullet$ in $\D_\qct(\FU)$ and $\CE^\bullet\oplus\CE^\bullet[1]$ in $\D_\qct(\FY)$ to get a perfect complex $\CP^\bullet$ in $\D_\qct(\FX)$. We do it by taking the third object in the following 
distinguished triangle on $\FX$:
\[
\CP^\bullet \lto
\R\gamma_*\CF^\bullet \oplus \R\delta_*(\CE^\bullet\oplus\CE^\bullet[1]) \lto
\R(\delta\beta)_*\beta^*(\CE^\bullet\oplus\CE^\bullet[1]) \overset{+}{\lto}
\]
where the middle map is the difference of the obvious morphisms.
Note that $\gamma^* \CP^\bullet \cong \CF^\bullet$ and $\delta^* \CP^\bullet \cong \CE^\bullet\oplus\CE^\bullet[1]$, therefore $\CP^\bullet$ is perfect.

To finish the proof define  $\CC^\bullet := \CP^\bullet \oplus \R\gamma_*\CQ^\bullet$. We claim that $\CC^\bullet$ is a generator of $\D_\qct(\FX)$. Let $\CG^\bullet \in \D_\qct(\FX)$ and suppose that $\Hom_{\D_\qct(\FX)}(\CC^\bullet[i], \CG^\bullet) = 0$ for all $i \in \ZZ$, let us see that necessarily $\CG^\bullet = 0$.
Note that
\[
\Hom_{\D_\qct(\FX)}(\CC^\bullet[i], \CG^\bullet) =
\Hom_{\D_\qct(\FX)}(\CP^\bullet[i], \CG^\bullet) \oplus
\Hom_{\D_\qct(\FX)}(\R\gamma_*\CQ^\bullet[i], \CG^\bullet).
\]

We use first that $\Hom_{\D_\qct(\FX)}(\R\gamma_*\CQ^\bullet[i], \CG^\bullet) = 0$ for all $i \in \ZZ$. Using the Mayer-Vietoris triangle
\[
\CG^\bullet \lto
\R\gamma_*\gamma^*\CG^\bullet \oplus \R\delta_*\delta^*\CG^\bullet \lto
\R(\gamma\alpha)_*(\gamma\alpha)^*\CG \overset{+}{\lto}
\]
we obtain isomorphisms
\begin{equation}
\begin{split}
\Hom_{\D_\qct(\FX)}(\R\gamma_*\CQ^\bullet[i], \R&\gamma_*\gamma^*\CG^\bullet) \oplus \Hom_{\D_\qct(\FX)}(\R\gamma_*\CQ^\bullet[i], \R\delta_*\delta^*\CG^\bullet) \liso \\
&\Hom_{\D_\qct(\FX)}(\R\gamma_*\CQ^\bullet[i], \R(\gamma\alpha)_*(\gamma\alpha)^*\CG)\\
\end{split}\label{inter}
\end{equation}
Note that
\[
\Hom_{\D_\qct(\FX)}(\R\gamma_*\CQ^\bullet[i], \R\delta_*\delta^*\CG^\bullet) \cong 
\Hom_{\D_\qct(\FY)}(\delta^*\R\gamma_*\CQ^\bullet[i], \delta^*\CG^\bullet) \]
therefore $\Hom_{\D_\qct(\FX)}(\R\gamma_*\CQ^\bullet[i], \R\delta_*\delta^*\CG^\bullet) = 0$ by \eqref{bcQ}. So, using isomorphisms \eqref{inter} we get
\[
\Hom_{\D_\qct(\FX)}(\R\gamma_*\CQ^\bullet[i], \R\gamma_*\gamma^*\CG^\bullet)  \cong 
\Hom_{\D_\qct(\FX)}(\R\gamma_*\CQ^\bullet[i], \R\gamma_*\R\alpha_*(\gamma\alpha)^*\CG),
\]
equivalently
\[
\Hom_{\D_\qct(\FU)}(\CQ^\bullet[i], \gamma^*\CG^\bullet)  \cong 
\Hom_{\D_\qct(\FV)}(\alpha^*\CQ^\bullet[i], (\gamma\alpha)^*\CG).
\]
Hence, $\Hom_{\D_\qct(\FU)}(\CQ^\bullet[i], \gamma^*\CG^\bullet) = 0$, for all $i \in \ZZ$. By Corollary \ref{bous}, the canonical map $\gamma^*\CG^\bullet \to \R\alpha_*(\gamma\alpha)^*\CG^\bullet$ is an isomorphism.
So 
\[
\R\gamma_*\gamma^*\CG^\bullet \cong \R\gamma_*\R\alpha_*(\gamma\alpha)^*\CG^\bullet \cong \R(\delta\beta)_*(\delta\beta)^*\CG^\bullet,
\] 
the last isomorphism by pseudo-functoriality. Using the previous Mayer-Vietoris triangle
\[
\CG^\bullet \lto
\R\gamma_*\gamma^*\CG^\bullet \oplus \R\delta_*\delta^*\CG^\bullet \lto
\R(\delta\beta)_*(\delta\beta)^*\CG^\bullet \overset{+}{\lto}
\]
it follows that 
\begin{equation}\label{Gdel}
\CG^\bullet \cong \R\delta_*\delta^*\CG^\bullet.
\end{equation}

Now we are going to use that $\Hom_{\D_\qct(\FX)}(\CP^\bullet[i], \CG^\bullet) = 0$ for all $i \in \ZZ$. We have the chain of isomorphisms
\begin{align*}
\Hom_{\D_\qct(\FX)}(\CP^\bullet[i], \CG^\bullet) &\cong 
\Hom_{\D_\qct(\FX)}(\CP^\bullet[i], \R\delta_*\delta^*\CG^\bullet) \tag{by (\ref{Gdel})}\\
&\cong \Hom_{\D_\qct(\FY)}(\delta^*\CP^\bullet[i], \delta^*\CG^\bullet)
\end{align*}
But $\delta^*\CP^\bullet = \CE^\bullet\oplus\CE^\bullet[1]$, therefore $\Hom_{\D_\qct(\FY)}(\CE^\bullet[i], \delta^*\CG^\bullet) = 0$ for all $i \in \ZZ$. Being $\CE^\bullet$ a generator of $\D_\qct(\FY)$, it follows that
$
\delta^*\CG^\bullet = 0.
$
As a consequence,
\[
\CG^\bullet \overset{\eqref{Gdel}}\cong 
\R\delta_*\delta^*\CG^\bullet = 0.
\]
and the proof is complete.
\end{proof}

A locally noetherian formal scheme $\FX$ is called \emph{semi-separated} if the diagonal map $\Delta_f \colon \FX \to \FX \times_{\spec{\ZZ}} \FX$ is an affine morphism where $f \colon \FX \to \spec{\ZZ}$ denotes the canonical map.
A useful criterion for a morphism of noetherian formal schemes to be affine is given by \cite[Proposition (10.16.2)]{ega1}.

\begin{cor}
Let $\FX$ be a noetherian semi-separated formal scheme. The category $\D_\qct(\FX)$ is a \emph{monogenic} algebraic stable homotopy category in the sense of \cite{hps}. 
\end{cor}

\begin{proof}
The fact that $\D_\qct(\FX)$ is an algebraic stable homotopy category is \cite[Corollary 6.19.]{ahst} and it is monogenic by the previous Theorem.
\end{proof}

\begin{rem}
This answers the question posed in \lc page 1251 (just before the references).
\end{rem}

\section{Appendix: compact objects form a skeletally small category}

Recall that a category $\SC$ is called \emph{skeletally small} or \emph{essentially small} if there is a set $S$ of objects of $\SC$ such that every object of $\SC$ is isomorphic to an object of $S$.

Let $(\FX, \CO_\FX)$ be a noetherian formal scheme. Denote by $\D_\qct(\FX)^\cc$ the full subcategory of $\D_\qct(\FX)$ formed by compact objects. As we remarked before it agrees with the full subcategory of $\D_\qct(\FX)$ formed by perfect complexes \cite[Proposition 6.12.]{ahst}.

\begin{prop} \label{afin-ss}
Let $\FX = \spf(A)$ be an affine formal scheme such that $A$ is an $I$-adic noetherian ring. The category $\D_\qct(\FX)^\cc$ is skeletally small.
\end{prop}

\begin{proof}
Perfect objects in this case are quasi-isomorphic to bounded complexes of finite type projective modules by the equivalence
\[
\D_\qct(\FX)^\cc \liso \D_{I}(A)^\cc
\]
induced by the equivalence recalled in the proof of lemma \ref{afin}. This assertion follows from Rickard's criterion, see \cite[Lemma 4.3.]{AJST}.

Now it is clear that finite type projective modules form a skeletally small subcategory of $A\md$ because a projective module is isomorphic to a direct summand of a free module. And the set of bounded complexes of these modules clearly can be represented by a set.
\end{proof}
 
\begin{thm} \label{sksm}
Let $(\FX, \CO_\FX)$ be a noetherian formal scheme. The category $\D_\qct(\FX)^\cc$ is skeletally small.
\end{thm}

\begin{proof}
We divide the proof in three steps, the first one will provide the basic tool that makes it possible to apply induction in later steps.

\textbf{Step 1}. Suppose that $\FX = \FU_1 \cup \FU_2$ where $\FU_1$ and $\FU_2$ are open formal subschemes of $\FX$. Assume that $\D_\qct(\FU_1)^\cc$, $\D_\qct(\FU_2)^\cc$ and $\D_\qct(\FU_1 \cap \FU_2)^\cc$ is skeletally small, we will see that $\D_\qct(\FX)^\cc$ is skeletally small.

Let us represent $\FX$ together with its open subsets and the corresponding canonical open embeddings by the diagram
\[
 \begin{tikzpicture}
       \draw[white] (0cm,2.5cm) -- +(0: \linewidth)
      node (G) [black, pos = 0.4] {$\FU_1 \cap \FU_2$}
      node (H) [black, pos = 0.6] {$\FU_1$};
      \draw[white] (0cm,0.5cm) -- +(0: \linewidth)
      node (E) [black, pos = 0.4] {$\FU_2$}
      node (F) [black, pos = 0.6] {$\FX$};
      \draw [->] (G) -- (H) node[above, midway, sloped, scale=0.75]{$k'$};
      \draw [->] (G) -- (E) node[left, midway, scale=0.75]{$j'$};
      \draw [->] (H) -- (F) node[right, midway, scale=0.75]{$j$};
      \draw [->] (E) -- (F) node[below, midway, scale=0.75]{$k$};
      \draw [->] (G) -- (F) node[auto, midway, scale=0.75]{$\ell$};
  \end{tikzpicture}
\]
Let $S_1$, $S_2$, $S_{12}$ be the sets that represent the classes of isomorphy of objets in $\D_\qct(\FU_1)^\cc$, $\D_\qct(\FU_2)^\cc$ and $\D_\qct(\FU_1 \cap \FU_2)^\cc$, respectively. Let $\CE^\bullet \in \D_\qct(\FX)^\cc$. There are complexes $\CF_1^\bullet \in S_1$, $\CF_2^\bullet \in S_2$, $\CF_{12}^\bullet \in S_{12}$ and isomorphisms
\[
\phi_1    \colon j^*\CE^\bullet    \liso \CF_1^\bullet,    \quad
\phi_2    \colon k^*\CE^\bullet    \liso \CF_2^\bullet,    \quad
\phi_{12} \colon \ell^*\CE^\bullet \liso \CF_{12}^\bullet. \quad
\] 
Note that we have a commutative triangle
\[
 \begin{tikzpicture}
       \draw[white] (0cm,2.25cm) -- +(0: \linewidth)
      node (G) [black, pos = 0.3] {${k'}^*\CF_1^\bullet$}
      node (H) [black, pos = 0.7] {${j'}^*\CF_2^\bullet$};
      \draw[white] (0cm,0.5cm) -- +(0: \linewidth)
      node (E) [black, pos = 0.5] {$\CF_{12}^\bullet$};
      \draw [->] (G) -- (H) node[above, midway, sloped, scale=0.75]{$\text{ via } \phi_2 \phi_1^{-1}$};
      \draw [->] (G) -- (E) node[auto, swap, midway, scale=0.75]{$\text{ via } \phi_{12} \phi_1^{-1}$};
      \draw [->] (H) -- (E) node[auto, midway, scale=0.75]{$\text{ via } \phi_{12} \phi_2^{-1}$};
  \end{tikzpicture}
\]
By adjunction and pseudo-functoriality, from the morphisms $\phi_1 \colon j^*\CE^\bullet \iso \CF_1^\bullet$, $\phi_2 \colon k^*\CE^\bullet \iso \CF_2^\bullet$ and $\phi_{12} \colon \ell^*\CE^\bullet \liso \CF_{12}^\bullet$,
we define
\[
\psi_1    \colon \R{}j_*\CF_1^\bullet  \lto \R\ell_*\CF_{12}^\bullet, \qquad
\psi_2    \colon \R{}k_*\CF_2^\bullet  \lto \R\ell_*\CF_{12}^\bullet.
\] 
So we get a map $\Psi \colon \R{}j_*\CF_1^\bullet \oplus \R{}k_*\CF_2^\bullet \lto \R\ell_*\CF_{12}^\bullet$, namely $\Psi := \binom{\psi_1}{-\psi_2}$. Putting all together, we construct a map of triangles
\[
 \begin{tikzpicture}
       \draw[white] (0cm,2.25cm) -- +(0: \linewidth)
      node (G) [black, pos = 0.1] {$\CG^\bullet$}
      node (H) [black, pos = 0.4] {$\R{}j_*\CF_1^\bullet \oplus \R{}k_*\CF_2^\bullet$}
      node (I) [black, pos = 0.75] {$\R\ell_*\CF_{12}^\bullet$}
      node (I+) [black, pos = 0.9] {};
      \draw[white] (0cm,0.5cm) -- +(0: \linewidth)
      node (E) [black, pos = 0.1] {$\CE^\bullet$}
      node (F) [black, pos = 0.4] {$\R{}j_*(j^*\CE^\bullet)\oplus\R{}k_*(k^*\CE^\bullet)$}
      node (S) [black, pos = 0.75] {$\R{}\ell_*(\ell^*\CE^\bullet)$}
      node (S+) [black, pos = 0.9] {};
      \draw [->] (G) -- (H) node[above, midway, sloped, scale=0.75]{};
      \draw [->] (H) -- (I) node[above, midway, sloped, scale=0.75]{$\Psi$};
      \draw [->] (I) -- (I+) node[above, midway, sloped, scale=0.75]{$+$};
      \draw [<-] (G) -- (E) node[left, midway, scale=0.75]{$\mu$};
      \draw [<-] (H) -- (F) node[left, midway, scale=0.75]{$\R{}j_*\phi_1\oplus\R{}k_*\phi_2$};
      \draw [->] (S) -- (I) node[left, midway, scale=0.75]{$\R{}\ell_*\phi_{12}$};
      \draw [->] (E) -- (F) node[below, midway, scale=0.75]{};
      \draw [->] (F) -- (S) node[above, midway, scale=0.75]{};
      \draw [->] (S) -- (S+) node[above, midway, scale=0.75]{$+$};
  \end{tikzpicture}
\]
The bottom row is the Mayer-Vietoris triangle of $\CE^\bullet$ with respect to the covering given by $\FU_1$ and $\FU_2$. The top row is the completion to a triangle of the map $\Psi$ and the right square is commutative by the definition of the maps. It follows that the map $\mu$ is an isomorphism in $\D_\qct(\FX)$. Note that $\CG^\bullet$ belongs to the set $S$ of those complexes that are cones of a map in which the source complex is a coproduct of an extension to $\FX$ of a complex in $S_1$ and a complex in $S_2$, and the target complex is an extension to $\FX$ of a complex in $S_{12}$. Therefore all perfect complexes in $\D_\qct(\FX)$ are isomorphic to a complex in $S$ as claimed.

\textbf{Step 2}. Suppose first that $\FX$ is separated and let us argue by induction on the number of affine open formal subschemes needed to cover $\FX$. Denote this number by $n$, therefore $\FX = \FU_1 \cup \dots \cup \FU_n$ with each $\FU_i$ an affine formal scheme for every $i \in \{1, \dots, n\}$. Let $\FU = \FU_1$ and $\FY = \FU_2 \cup \dots \cup \FU_n$. Note that
\[
\FS := \FU \cap \FY = \FU \cap (\FU_2 \cup \dots \cup \FU_n) =
(\FU \cap \FU_2) \cup \dots \cup (\FU \cap \FU_n).
\] 
But, being $\FX$ separated, the open subschemes $(\FU \cap \FU_2), \dots, (\FU \cap \FU_n)$ are affine therefore $\FS$ can be covered by $n - 1$ affine open subschemes. Now $\FU$ satisfies the thesis of the theorem by Lemma \ref{afin-ss} and both $\FY$ and $\FS$ by induction hypothesis. By Step 1, the result follows for $\FX$.

\textbf{Step 3}.
Finally we allow $\FX$ to be any noetherian formal scheme. Assume that $\FX = \FU_1 \cup \dots \cup \FU_n$ where each $\FU_i$ is a separated noetherian formal scheme. We will use induction in the number of \emph{separated} open formal sub\-schemes needed to cover $\FX$. With the previous notation $\FU = \FU_1$ satisfies the thesis of our theorem by Step 2. Both $\FY$ and $\FS$ can be covered by $n - 1$ separated open formal subschemes because the intersection of open separated subschemes is a separated formal scheme. Therefore, they also satisfy the thesis of our theorem by the induction hypothesis. Using Step 1 again the proof is complete.
\end{proof}

\begin{rem}
The issue of the skeletal smallness of the category of finite type locally free sheaves on any locally ringed space was dealt with in \cite[(\textbf{0}.5.6.1)]{ega1}. The case of perfect complexes on a quasi-compact, quasi-sepa\-ra\-ted scheme was mentioned in \cite[Appendix F]{tt}. A proof in this case along the lines of the previous Theorem may be given simplifying somehow the strategy of \lc (3.20.4) -- (3.20.6). We leave the details to the interested reader.
\end{rem}



\begin{thebibliography}{ABCDEF}


\bibitem[AJL1]{dfs} Alonso Tarr{\'{\i}}o, L.; Jerem{\'{\i}}as
L{\'o}pez, A.; Lipman, J.: Duality and flat base change on formal schemes, in
\textit{Studies in duality on noetherian formal schemes and non-noetherian
ordinary schemes}. Providence, RI: American Mathematical Society. Contemp.
Math. \textbf{244}, 3--90 (1999).

\bibitem[AJL2]{fgm} Alonso Tarr{\'{\i}}o, L.; Jerem{\'{\i}}as
L{\'o}pez, A.; Lipman, J.: Greenlees-May duality of formal schemes, in
\textit{Studies in duality on noetherian formal schemes and non-noetherian
ordinary schemes}. Providence, RI: American Mathematical Society. Contemp.
Math. \textbf{244}, 93--112 (1999).

\bibitem[AJS1]{AJS} Alonso Tarr{\'{\i}}o, L.; Jerem{\'{\i}}as
L{\'o}pez, A.; Souto Salorio, M. Jos\'e: Localization in categories of
complexes and unbounded resolutions \textit{Canad. J. Math.} \textbf{52}
(2000), no. 2,  225--247.

\bibitem[AJS2]{AJST} Alonso Tarr{\'{\i}}o, L.; Jerem{\'{\i}}as
L{\'o}pez, A.; Souto Salorio, M. Jos\'e: Construction of $t-$structures and equivalences of derived categories \textit{Trans. Amer. Math. Soc.} \textbf{355} (2003),  2523--2543.

\bibitem[AJS3]{AJSB} Alonso Tarr{\'{\i}}o, L.; Jerem{\'{\i}}as
L{\'o}pez, A.; Souto Salorio, M. Jos\'e: Bousfield localization on formal schemes. \textit{J. Algebra} \textbf{278} (2004), no. 2, 585--610.


\bibitem[AJPV]{ahst} Alonso Tarr{\'{\i}}o, L.; Jerem{\'{\i}}as
L{\'o}pez, A.; P\'erez Rodr\'{\i}guez, M.; Vale Gonsalves M. J.:
The derived category of quasi-coherent sheaves and axiomatic stable homotopy. 
\textit{Adv. Math.} \textbf{218} (2008) 1224--1252.

\bibitem[BN]{BN} B\"okstedt, M.; Neeman, A.: Homotopy limits in
triangulated categories, {\it Compositio Math.\ }{\bf 86}~(1993),
209--234.

\bibitem[BvdB]{bb} Bondal, A.; van den Bergh, M.: Generators and representability of functors in commutative and noncommutative geometry, \textit{Mosc. Math. J.} \textbf{3} (2003), no. 1, 1--36, 258.



Second edition. John Wiley \& Sons, Chichester, 1982.

\bibitem[DG]{dg} Dwyer, W. G.; Greenlees, J. P. C.: Complete modules and torsion modules. \textit{Amer. J. Math.} \textbf{124} (2002), no. 1, 199--220.






\bibitem[Gr1]{toh} Grothendieck, A.: Sur quelques points d'alg\`ebre
homologique. \textit{T\^ohoku Math. J.} (2), \textbf{9} (1957),
119--221. 

\bibitem[Gr2]{GB} Grothendieck, A.: 
G\'eom\'etrie formelle et g\'eom\'etrie alg\'ebrique. 
\textit{S\'eminaire Bourbaki}, \textbf{5} (1958-1960), Exp. No. 182, 28 pp. 

\bibitem[EGA I]{ega1} Grothendieck, A.; Dieudonn\'{e}, J. A.: 
\textit{El\'{e}ments de  G\'{e}om\'{e}trie Alg\'{e}brique I}, Grundlehren
der math. Wiss. {\bf 166}, Springer--Verlag, Heidelberg, 1971.



\bibitem[HPS]{hps} Hovey, M.; Palmieri, J. H.; Strickland, N. P.: Axiomatic stable homotopy theory.  \textit{Mem. Amer. Math. Soc.} \textbf{128} (1997), no. 610.

\bibitem[I]{I1} Illusie, L.: G\'en\'eralit\'es sur las conditions de finitude dans les cat\'egories d\'eri\-v\'ees, in {\it
Th\'eorie des Intersections et Th\'eor\`eme de Riemann-Roch (SGA~6),}
Lecture Notes in Math., no.\,{\bf 225}, Springer-Verlag, New York, 1971,
78--159.




\bibitem[L]{LDC} Lipman, J.: {\it Notes on Derived Functors and Grothendieck Duality}. Lecture Notes in Math., Springer-Verlag, New York, \emph{to appear}.
Preprint at
\texttt{http://www.math.purdue.edu/\~{}lipman/}.

\bibitem[LNe]{LN} Lipman, J.; Neeman, A.:
Quasi-perfect scheme-maps and boundedness of the twisted inverse image functor, \textit{Illinois Journal of Mathematics}, \textbf{51} (2007), no. 1, 209--236.

\bibitem[LNS]{LNS} Lipman, J.; Nayak, S.; Sastry P.: Variance and duality for Cousin complexes on formal schemes, pp. 3--133 in {\it Pseudofunctorial behavior of Cousin complexes on formal schemes}. Contemp. Math., \textbf{375}, Amer. Math. Soc., Providence, RI, 2005.



\bibitem[Ne1]{Ntty}Neeman, A.: The connection between the $K$-theory
localization theorem of Thomason, Trobaugh and Yao and the smashing
subcategories of Bousfield and Ravenel. {\it Ann. Sci. \'Ecole Norm. Sup.}
(4) {\bf 25} (1992), no.~5, 547--566. 

\bibitem[Ne2]{Ngd} Neeman, A.: The Grothendieck duality theorem via
Bousfield's techniques and Brown representability. {\it J. Amer. Math. Soc.}
{\bf 9} (1996), no. 1, 205--236.

\bibitem[Ne3]{Ntc} Neeman, A.: \textit{Triangulated categories}. Annals of
Mathematics Studies, \textbf{148}. Princeton University Press, Princeton, NJ,
2001.






\bibitem[TT]{tt} Thomason, R. W.; Trobaugh, T.: Higher algebraic $K$-theory of schemes and of derived categories. \textit{The Grothendieck Festschrift, Vol. III}, 247--435, Progr. Math., {\bf 88}, Birkh\"auser Boston, Boston, MA, 1990.




\end{thebibliography}
\end{document}